\documentclass[12pt]{amsart} 

\usepackage{amsmath} \usepackage{amssymb} \usepackage{amsthm}
\usepackage{amscd} \usepackage{graphicx}

\def\fkp{{\mathfrak{p}}}

\def\opn#1#2{\def#1{\operatorname{#2}}} 
\opn\chara{char} \opn\length{\ell}  
\opn\embdim{emb\,dim} 
\opn\dim{dim} \opn\codim{codim} \opn\height{height} 
\opn\indeg{indeg} \opn\reg{reg} 
\opn\projdim{proj\,dim} \opn\injdim{inj\,dim} 
\opn\depth{depth}  \opn\grade{grade}
\opn\Ass{Ass} \opn\Min{Min} \opn\Assh{Assh}
\opn\Hom{Hom} \opn\Ker{Ker} \opn\Im{Im} \opn\Coker{Coker} \opn\rank{rank}
\opn\Ann{Ann}
\opn\Spec{Spec} \opn\Proj{Proj}
\opn\Div{div}
\opn\F{F}
\opn\PF{PF}
\opn\t{t}
\opn\g{g}
\opn\e{e}
\opn\SG{SG}
\opn\m{m}
\opn\G{G}
\opn\r{r}
\opn\and{and}
\opn\Card{Card}
\opn\n{n}
\opn\min{min}
\opn\Maximals{Maximals}
\opn\Ap{Ap}
\opn\max{max}
\opn\S{S}
\opn\mod{mod}
\opn\gcd{gcd}
\opn\L{L}
\opn\N{N}

\theoremstyle{plain}
\newtheorem{thm}{Theorem}[section] 
\newtheorem{prop}[thm]{Proposition}
\newtheorem{cor}[thm]{Corollary}
\newtheorem{lemma}[thm]{Lemma}

\newtheorem{propdef}[thm]{Proposition-Definition}
\theoremstyle{definition} 
\newtheorem{defn}[thm]{Definition}
\newtheorem{exam}[thm]{Example} 
\theoremstyle{remark}
\newtheorem{remark}[thm]{Remark}

\newtheorem*{acknowledgement}{Acknowledgment}

\title{Symmetries on almost symmetric numerical semigroups}
\author{Hirokatsu Nari}
\address{Graduate School of Integrated Basic Sciences, Nihon University, Setagaya-ku, Tokyo, 156-0045, JAPAN}
\email{hirokatsu1022@gmail.com}

\date{\today}
\subjclass[2000]{Primary 20M14, Secondary 20M25, 13F99 }
\keywords{numerical semigroup, almost symmetric numerical semigroup, dual of maximal ideal, gluing of numerical semigroups}
\begin{document}

\begin{abstract}
The notion of almost symmetric numerical semigroup was given by 
V. Barucci and R. Fr\"oberg in \cite{BF}.
We characterize almost symmetric numerical semigroups by symmetry of pseudo-Frobenius numbers.
We give a criterion for $H^*$ (the dual of $M$) to be almost symmetric numerical semigroup.
Using these results we give a formula for multiplicity of an opened modular numerical semigroups.
Finally, we show that if $H_1$ or $H_2$ is not symmetric, 
then the gluing of $H_1$ and $H_2$ is not almost symmetric.

\end{abstract}

\maketitle

\section{Introduction}
Let $\mathbb{N}$ be the set of nonnegative integers. A \textit{numerical semigroup} 
 $H$ is a subset of $\mathbb{N}$ which is closed under addition,
 contains the zero element and whose complement in $\mathbb{N}$ is finite.\par
 Every numerical semigroup $H$ admits a finite system of generators, that is, there exist $a_1, ..., a_n\in H$ such that 
 $H=\langle a_1, ..., a_n\rangle =\{\lambda_1a_1+\cdot \cdot \cdot +\lambda _na_n\mid \lambda _1, ..., \lambda _n\in \mathbb{N}\}$. \par
 Let $H$ be a numerical semigroup and
 let $\{a_1<a_2<\cdot \cdot \cdot <a_n\}$ be its minimal generators.
 We call $a_1$ {\it the multiplicity} of $H$ and denote it by $\m(H)$, and we call $n$ 
 {\it the embedding dimension} of $H$ and denote it by $\e(H)$. In general, $\e(H)\leq \m(H)$. 
 We say that $H$ has {\it maximal embedding dimension} if $\e(H)=\m(H)$. The set   
 $\G(H):=\mathbb{N}\setminus H$ is called {\it the set of gaps} of $H$.
 Its cardinality is said to be {\it the genus} of $H$ and we denote it by 
 $\g(H)$. \par
 If $H$ is a numerical semigroup, the largest integer in $\G(H)$ is called 
 {\it Frobenius number} of $H$ and we denote it by $\F(H)$.
 It is known that $2\g(H)\geq \F(H)+1$.
 We say that $H$ is {\it symmetric} if for every $z \in \mathbb{Z}$, either $z\in H$ or $\F(H)-z \in H$, or equivalently, $2\g(H)=\F(H)+1$.
We say that $H$ is {\it pseudo-symmetric} if for every $z \in \mathbb{Z}$, $z \not =\F(H)/2$, either $z\in H$ or $\F(H)-z \in H$,
or equivalently, $2\g(H)=\F(H)+2$. \par
We say that an integer $x$ is a {\it pseudo-Frobenius number} of $H$ if 
$x \not \in H$ and $x+h \in H$ for all $h \in H, h\ne 0$.
 We denote by $\PF(H)$ the set of pseudo-Frobenius numbers of $H$.
 The cardinality in $\PF(H)$ is called the {\it type} of $H$, denoted by $\t(H)$.
 Since $\F(H) \in \PF(H)$, $H$ is symmetric if and only if $\t(H)=1$. \par
 This paper studies almost symmetric numerical semigroups. 
The concept of almost symmetric numerical semigroup was introduced by V. Barucci and R. Fr\"oberg \cite{BF}.
They developed a theory of almost symmetric numerical semigroups and gave many results (see \cite{Ba}, \cite{BF}).
This paper aims at an alternative characterization of almost symmetric numerical semigroups.
(see Theorem \ref{almsymm}). \par
In \cite{BF} the authors proved that $H$ is almost symmetric and has maximal embedding dimension if and only if
$H^* = M-M$ (the dual of $M$) is symmetric,
where $M$ denotes the maximal ideal of $H$.
In Section 3 we will study the problem of when $H^*$ is an almost symmetric numerical semigroup.\par
The notion of opened modular numerical semigroup was introduced by 
J. C. Rosales, and J. M. Urbano-Blanco \cite{RU}.
In section 4 we will give a formula for multiplicity of an opened modular numerical semigroups.
Also, we prove that opened modular numerical semigroups are almost symmetric.\par
Proportionally modular and symmetric numerical semigroups generated by three elements were investigated by 
J. C. Rosales, P. A. Garc\'ia-S\'anchez and J. M. Urbano-Blanco in \cite{RGU2}.
In section 5 we will study the proportionally modular and pseudo-symmetric numerical semigroups generated by three elements.\par
Let $H=\langle a_1, a_2, \ldots, a_n \rangle$ be a numerical semigroup.
For a fixed field $k$ and a variable $T$ over $k$, let $R=k[H]=k[T^{a_1}, T^{a_2}, \ldots, T^{a_n}]$ 
be the semigroup ring of $H$.
We say that $H$ is a {\it complete intersection} if the semigroup ring $k[H]$ is a complete intersection.
The notion of gluing of numerical semigroups was introduced in C. Delorme \cite{De},
he proved that a numerical semigroup is a complete intersection if and only if it is a gluing of two complete 
intersection numerical semigroups, and gave many interesting results (see \cite{De} 10. Proposition.).
In the last section 6 we show that for two numerical semigroups $H_1$ and $H_2$,
if $H_1$ or $H_2$ is not symmetric, then the gluing of $H_1$ and $H_2$ is not almost symmetric.

 \section{Almost symmetric numerical semigroups}
 Let $H$ be a numerical semigroup and let $n$ be one of its nonzero elements.
 We define 
 \[\Ap(H, n)=\{h \in H \mid h-n \not \in H\}.\]
 This set is called the Ap\'ery set of $h$ in $H$.
 By definition, $\Ap(H, n)=\{0=w(0), w(1), \ldots, w(n-1) \}$, where $w(i)$ is the least element of $H$ congruent with $i$ modulo $n$,
 for all $i \in \{0, \ldots, n-1\}$.
 We can get pseudo-Frobenius numbers of $H$ from the Ap\'ery set by the following way:
Over the set of integers we define the relation $ \leq_H$, that is, $a  \leq_H b$ implies that $b-a \in H$.
Then we have the following result (see \cite{RG} Proposition 2.20).\par

\begin{prop}
 Let $H$ be a numerical semigroup and let $n$ be a nonzero element of $H$.
Then
\[\PF(H)=\{ \omega -n \mid \omega \ \text{is maximal with respect to} \leq_H \text{in} \ \Ap(H, n) \}.\]
\end{prop}

 It is easy to check that $\F(H)=\max\Ap(H, n)-n$
 and $\g(H)=\frac{1}{n}\sum_{h \in \Ap(H, n)}h-\frac{n-1}{2}$ (see \cite{RG} Proposition 2.12). \par
 Let $H$ be a numerical semigroup. A {\it relative ideal $I$} of $H$ is a subset of $\mathbb{Z}$ such that $I+H\subseteq I$
 and $h+I=\{h+i \mid i \in I\} \subseteq H$ for some $h \in H$. An {\it ideal} of $H$ is a relative ideal of $H$ with $I\subseteq H$.
 It is straightforward to show that if $I$ and $J$ are relative ideals of $H$, then $I-J:=\{z \in \mathbb{Z} \mid z+J \subseteq  I\}$ 
 is a also relative ideal of $H$.
 The ideal $M:=H \setminus \{0\}$ is called {\it the maximal ideal} of $H$.
 We easily deduce that $M-M=H \cup \PF(H)$.
 We define 
 \[K = K_H:=\{\F(H)-z \mid z \not \in H \}.\]
 It is clear that $H \subseteq  K$ and $K$ is a relative ideal of $H$. 
 This ideal is called {\it the canonical ideal} of $H$.\par
 We define $\N(H) := \{h \in H \mid h < \F(H)\}$.
 We already know that if $h \in \N(H)$, then $\F(H) - h \not \in H$, and 
 if $f \in \PF(H)$, $\not = \F(H)$, then $\F(H) - f \not \in H$.
 Then the map
 
\[
\begin{array}{ccc}
\N(H) \cup [\PF(H) \setminus \{\F(H)\}] & {\longrightarrow } & \G(H) \\
\rotatebox{90}{$\in$} & & \rotatebox{90}{$\in$} \\
h & \longmapsto & \F(H) - h
\end{array}
\]
is injective, which proves the following.

\begin{prop} \label{alm-ineq}
Let $H$ be a numerical semigroup. Then
\[2\g(H) \geq \F(H) + \t(H).\]
\end{prop}

Clearly, if a numerical semigroup is symmetric or pseudo-symmetric, then the equality of Proposition \ref{alm-ineq} holds. 
In general, a numerical semigroup is called almost symmetric if the equality holds.

 \begin{propdef} \label{defalm} \cite{Ba} \cite{BF}
 Let $H$ be a numerical semigroup. Then the following conditions are equivalent.
 \begin{enumerate}
 \item $K_H\subset M-M$. \par
 \item $z \not \in H$ implies that either $\F(H) - z \in H$ or $z \in \PF(H)$. \par
 \item $2\g(H)=\F(H)+\t(H)$. \par
 \item $K_{M - M} = M - \m(H)$.
 \end{enumerate}
 \par
 A numerical semigroup $H$ satisfying either of these equivalent conditions is said to be 
 {\it almost symmetric}.
 \end{propdef}
 \par
 
 It is easy to show that if $H$ is symmetric or pseudo-symmetric, then $H$ is almost symmetric.
 Conversely, an almost symmetric numerical semigroups with type two is pseudo-symmetric (see Corollary \ref{ps}). \par
 We now give a characterization of almost symmetric numerical semigroups by symmetry of pseudo Frobenius numbers.
 
 \begin{thm} \label{almsymm} Let $H$ be a numerical semigroup and let $n$ be one of its nonzero elements.
 Set $\Ap(H, n)=\{0<\alpha_1< \cdot \cdot \cdot <\alpha_m\}
                      \cup \{\beta_1<\beta_2< \cdot \cdot \cdot <\beta_{\t(H)-1}\}$ with $m=n-\t(H)$ and 
 $\PF(H)=\{\beta_i-n, \alpha_m-n=\F(H) \mid 1\leq i\leq \t(H)-1\}$.
 We put $f_i=\beta_i-n$ and 
 $f_{\t(H)}=\alpha_m-n=\F(H)$. Then the following conditions are equivalent. \par
 
 \begin{enumerate}
 \item $H$ is almost symmetric. \par
 \item $\alpha_i+\alpha_{m-i}=\alpha_m$ for all $i \in \{1, 2, \ldots, m-1\}$
 and $\beta_j+\beta_{\t(H)-j}=\alpha_m+n$ for all $j \in \{1, 2, \ldots, \t(H)-1\}$.\par
 \item $f_i+f_{\t(H)-i}=\F(H)$ for all $i \in \{1, 2, \ldots, \t(H)-1\}$. \par
 \end{enumerate} 
 \end{thm}
 \par
 
 \begin{proof} For simplicity, we put $t=\t(H)$.\par
 $(1) \Longrightarrow (2)$. Since $\alpha _i - n \not \in H$,
 $\F(H) - (\alpha _i - n) = \alpha _m - \alpha_i \in H$ and $\alpha_m - (\alpha_i - n) \not \in H$,
 by \ref{defalm} (2).
 Hence $\alpha_m - \alpha_i \in \Ap(H, n)$.
 If $\alpha_m - \alpha_i = \beta_j$ for some $j$, then $\F(H) = \alpha_i + f_j \in H$.
 Hence we have that $\alpha_i+\alpha_{m-i}=\alpha_m$ for all $i \in \{1, 2, \ldots, m - 1\}$.
 Next, we see that $\beta_j+\beta_{t-j}=\alpha_m+\m(H)$ for all $j \in \{1, 2, \ldots, t-1\}$.
 Since $\alpha_m-\beta_j = \F(H)-f_j \not \in H$,
 by \ref{defalm} (2) we get $\alpha_m - \beta_j \in \PF(H)$,
 that is, $\alpha_m - \beta_j = \beta_{\t(H) - j} - n$ for all $j \in \{1, 2, \ldots, t-1\}$.\par
 $(2) \Longrightarrow (3)$. By hypothesis, $(\beta_j-n)+(\beta_{t-j}-n)=\alpha_m-n$ 
 implies $f_j+f_{t-j}=\F(H)$.\par
 $(3) \Longrightarrow (1)$. In view of Proposition-Definition \ref{defalm},
 it suffices to prove that $K \subset M-M$.
 Let $x \in K$ and $x = \F(H) - z$ for some $Z \not \in H$.
 If $z \in \PF(H)$, then $x \in \PF(H)$ by condition (3).
 If $z \not \in \PF(H)$, then $z + h \in \PF(H)$ for some $h \in M$.
 Then $x = \F(H) - (z + h) + h \in H$, since $\F(H) - (z + h) \in \PF(H)$.
 Hence we have that $H$ is almost symmetric.
 \end{proof}
 
 \begin{remark}
  When $H$ is symmetric or pseudo-symmetric, the equivalence of (1) and (2) is shown Proposition 4.10 and 4.15 of \cite{RG} 
 \end{remark}
 
 \begin{exam}\label{alm-exam} (1) Let $H=\langle 5, 8, 11, 12 \rangle$. Then 
 $\Ap(H, 5)=\{0, 8, 11, 12, 16\}$ and $\PF(H)=\{6, 7, 11\}$, we see from Theorem \ref{almsymm} (3) that $H$ is not almost symmetric. \\
 (2) Let $a$ be an odd integer greater than or equal to three and let $H= \langle a, a+2, a+4, \ldots, 3a-2 \rangle$.
 $H$ has maximal embedding dimension, so that $\PF(H)=\{2, 4, \ldots, 2(a-1) \}$.
 Hence we get $H$ is almost symmetric.
 \end{exam}
 \par
 
 We obtain the following corollary from Theorem \ref{almsymm} (3). 
 
 \begin{cor} \label{ps} Let $H$ be a numerical semigroup. Then $H$ is almost symmetric with $\t(H)=2$
 if and only if $H$ is pseudo-symmetric. 
 \end{cor}
 \par
 
 \section{When is $H^*$ almost symmetric ?}
 Let $H$ be a numerical semigroup with maximal ideal $M$. 
 If $I$ is a relative ideal of $H$, then relative ideal $H-I$ is called {\it the dual of $I$ with respect to $H$}.
 In particular, the dual of $M$ is denoted by $H^*$.\par
 For every relative ideal $I$ of $H$, $I-I$ is a numerical semigroup.
 Since $H^*=H-M=M-M$, $H^*$ is numerical semigroup.
 By definition, it is clear that $\g(H^*)=\g(H)-\t(H)$. \par
 In \cite{BF} the authors solved the problem of when the dual of $M$ is a symmetric.\par
 
 \begin{thm}\label{H*:symm} \cite{BF} Let $H$ be a numerical semigroup. Then
 $H$ is almost symmetric and maximal embedding dimension if and only if $H^*$ is symmetric.
 \end{thm}
 
 \begin{exam} On the Example \ref{alm-exam} (2),
 $H = \langle a, a+2, a+4, \ldots, 3a-2 \rangle$ has maximal embedding dimension and almost symmetric.
 Hence we have that $H^* = H \cup \{2, 4, \ldots, 2(a-1) \} = \left< 2, a \right>$ is symmetric.
 \end{exam}
  
 In this section we will ask when is $H^*$ almost symmetric in general case (see Theorem \ref{m:t+t*}).
 Surprisingly, using our criterion for $H^*$ to be almost symmetric Theorem \ref{H*:symm} can be easily seen. \par
 Let $H$ be a numerical semigroup.
 Then we set 
 \[\L(H):= \{a \in H \mid a-\m(H) \not \in H^* \}.\]
 By definition we have that $\Card\L(H) = \m(H) - \t(H)$
 and $\Ap(H, \m(H))=\{f+\m(H) \mid f \in \PF(H) \} \cup \L(H)$.
 We describe $\Ap(H^*, \m(H))$ in terms of $\PF(H)$ and $\L(H)$.

 \begin{lemma} \label{apery} Let $H$ be a numerical semigroup. Then
 \[\Ap(H^*, \m(H))=\PF(H) \cup \L(H).\]
 \end{lemma}
 \par
 
 \begin{proof} 
 Since $H^*=H \cup \PF(H)$,
 clearly $\Ap(H^*, \m(H))\supseteq  \PF(H) \cup \L(H)$.\par
 Conversely we take $a \in \Ap(H^*, \m(H))$ and $a \not \in \PF(H)$.
 Then $a \in H$ and $a-\m(H) \not \in H^*$.
 Hence we have that $a \in \PF(H) \cup \L(H)$.
 \end{proof}
 \par
 
 By Lemma \ref{apery}, the Frobenius number of $H^*$ is easy to compute. \par
 
 \begin{prop} \cite{BDF} \label{F(H^*)} Let $H$ be a numerical semigroup. Then 
 \[ \F(H^*)=\F(H)-\m(H). \]
 \end{prop}
 \par
 
 \begin{proof}
 Clearly $\F(H) - \m(H) \not \in H^*$, by Lemma \ref{apery}.
 Let $x > \F(H) - \m(H)$ and $h \in M$.
 Then $x + h > \F(H) - \m(H) + h \geq \F(H)$,
 thus we get $\F(H^*) = \F(H) - \m(H)$.
 \end{proof}
 
 Every numerical semigroup is dual of maximal ideal for some numerical semigroup.
 
 \begin{prop}\label{T^*} Let $H$ be a numerical semigroup. Then there exists a numerical semigroup $T \subset H$ such that $T^*=H$.
 \end{prop}
 
 \begin{proof} Let $\Ap(H, h)=\{0< \alpha_1 < \cdots < \alpha_{h-1}\}$ for some $h \in H$.
 We put $T=\langle h, h+\alpha_1, \ldots, h+\alpha_{h-1} \rangle$.
 Since $T$ has maximal embedding dimension, $\PF(T)=\{\alpha_1 < \cdots < \alpha_{h-1}\}$.
 Hence we get $T^*=T \cup \PF(T)=H$.
 \end{proof}
 
 \begin{remark}In Proposition \ref{T^*}, such numerical semigroup $T$ is not determined uniquely.
 Indeed, we put $H_1= \langle 5, 6, 8, 9 \rangle$ and $H_2= \langle 3, 7, 8 \rangle$.
 Then $\PF(H_1)=\{3, 4, 7 \}$ and $\PF(H_2)=\{4, 5 \}$.
 Therefore we have $H_1^*=H_2^*= \langle 3, 4, 5 \rangle$.
 \end{remark}

 The following is the main Theorem of this section.
 
 \begin{thm}\label{m:t+t*} Let $H$ (resp. $H^*$) be an almost symmetric numerical semigroup. Then
 $H^*$ (resp. $H$) is an almost symmetric if and only if $\m(H)=\t(H)+t(H^*)$.
 \end{thm}
 
 \begin{proof}If $H$ is almost symmetric, then
 \begin{align*}
 2\g(H^*) & = 2\g(H) - 2\t(H) \\
          & = \F(H) - \t(H) \\
          & = \F(H^*) + \m(H) - \t(H). \ \  \text{(by Proposition \ref{F(H^*)})}
 \end{align*}
 
 If $H^*$ is almost symmetric, then
 \begin{align*}
 2\g(H) & = 2\g(H^*) + 2\t(H) \\
        & = \F(H^*) + \t(H^*) + 2\t(H) \\
        & = \F(H) + 2\t(H) + \t(H^*) - \m(H). \ \ \text{(by Proposition \ref{F(H^*)})}
 \end{align*}
 
 Observing these inequalities, we deduce the assertion.
 \end{proof}
 
 Using Theorem \ref{m:t+t*} we prove Theorem \ref{H*:symm}. 
 
 \vspace{2mm}
 {\it Proof of Theorem \ref{H*:symm}.}
 We assume that $H$ is almost symmetric and maximal embedding dimension.
 Then $\m(H)=\t(H) + 1$.
 Hence we have
 \begin{align*}
 \t(H^*) & \leq 2\g(H^*) - \F(H^*) \ \ \text{(by Proposition \ref{alm-ineq})} \\
        &= 2\g(H) - 2\t(H) - (\F(H) - \m(H)) \ \ \text{(by Proposition \ref{F(H^*)})} \\
        &= \m(H) - \t(H) \\
        &= 1.
 \end{align*}
 This implies $H^*$ is symmetric.\par
 Conversely, let $H^*$ be symmetric. 
 By Theorem \ref{m:t+t*}, it is enough to show that $\m(H) = \t(H) + 1$.
 We assume $\m(H) > \t(H) + 1$. 
 Then 
 \begin{align*}
 2\g(H^*) -\F(H^*) & = 2\g(H) - 2\t(H) - (\F(H) - \m(H)) \ \ \text{(by Proposition \ref{F(H^*)})}\\
                   & \geq \m(H) - \t(H) \\
                   & > 1.
 \end{align*}
 Since $H^*$ is symmetric, this is a contradiction.
 Thus we get $H$ is almost symmetric and maximal embedding dimension. \hspace{8cm} $\square$ 
 \vspace{2mm}

 Let $H=\langle a_1, a_2, ..., a_n \rangle$ be an almost symmetric numerical semigroup with $a_1< a_2 < \ldots < a_n$.
 If $\e(H) = n = a_1$ (that is, $H$ has maximal embedding dimension), 
 then the maximal element of $\Ap(H, a_1)$ is equal to $a_n$. 
 If $n < a_1$, then the maximal element of $\Ap(H, a_1)$ is greater than $a_n$.
  
 \begin{lemma}\label{key} Let $H=\langle a_1, a_2, ..., a_n \rangle$ be a numerical semigroup and let $n<a_1$.
 If $H$ is almost symmetric, then $\max\Ap(H, a_1) \not =a_n$. 
 \end{lemma}
 
 \begin{proof}
 We assume $\max\Ap(H,a_1)=a_n$. Since $H$ is almost symmetric, by Theorem \ref{almsymm} we have that
 \[\Ap(H,a_1)=\{0<\alpha_1<\cdot \cdot \cdot <\alpha_m<a_n\} \cup \{\beta_1<\cdot \cdot \cdot <\beta_{a_1-m-2}\},\]
 where $\alpha_i+\alpha_{m-i+1}=a_n$ for all $i \in \{1, 2, \ldots, m \}$
 and $\PF(H)=\{\beta_1-a_1<\cdot \cdot \cdot <\beta_{a_1-m-2}-a_1<a_n-a_1\}$.
 Since $\e(H)<\m(H)$, there exist $i$ such that $a_i=\alpha_j$ for some $j$.
 Hence we get $a_n=a_i+\alpha_k$ for some $k$. But this is a contradiction, because $a_n$ is a minimal generator of $H$. 
 \end{proof}
 
 \begin{prop}\label{type-H*} Let $H$ be an almost symmetric numerical semigroup with $\e(H) < \m(H)$. Then the following conditions hold:
 \begin{enumerate}
 \item $\e(H) + 1 \leq \t(H) + \t(H^*) \leq \m(H)$, \par
 \item $\t(H^*) \leq \e(H)$.
 \end{enumerate}
 \end{prop}
 
 \begin{proof} (1) First, we show that $\t(H) + \t(H^*) \leq \m(H)$.
 Since $H$ is almost symmetric, we get 
 \begin{align*}
 2\g(H^*) & = \F(H^*) + \m(H) - \t(H)\\
          & \geq \F(H^*) + \t(H^*) \ \ \text{(by Proposition \ref{alm-ineq})}.
 \end{align*}
 This inequality means $\t(H) + \t(H^*) \leq \m(H)$.
 Next, we prove $\e(H) + 1 \leq \t(H) + \t(H^*)$. Assume that $H = \left< a_1, \ldots, a_n \right>$ and $\m(H) = a_1$.
 Put $\PF(H) = \{f_1 < \cdots < f_{\t(H) - 1} < \F(H) \}$.
 By Lemma \ref{key}, $\F(H) + a_1 \not = a_i$ for all $i \in \{2, \cdots, a_1 - 1\}$.
 Also we have that for any $j \in \{1, \ldots, \t(H) - 1\}$, $f_j \not \in \PF(H^*)$ 
 by the symmetries of the pseudo-Frobenius numbers of $H$.
 This means
 \[0 \leq k := \Card\{a_i \mid a_i - a_1 \in \PF(H)\} \leq \t(H) - 1.\]
 Hence we have the inequality
 \[\e(H) - (\t(H) - 1) \leq \e(H) - k \leq \t(H^*).\] 
 (2) Let $H = \left< a_1, \ldots, a_n \right>$.
 It is enough to show that $\PF(H^*) \subseteq  \{\F(H) - a_i \mid 1 \leq i \leq n \}$.
 Take $x \in \PF(H^*)$. Since $x \not \in H^*$, we get $\F(H) - x \in H$ by \ref{defalm} (2).
 We assume $\F(H) - x \in 2M$, where $M$ denotes the maximal ideal of $H$.
 Then there exist $h \in M$ such that $\F(H) - x = a_i + h$ for some $a_i$, this means $\F(H) \in H$, a contradiction.
 Hence we have $\F(H) - x \not \in 2M$, that is, $\F(H) - x = a_i$ for some $i$.
 Thus we obtain that $\PF(H^*) \subseteq  \{\F(H) - a_i \mid 1 \leq i \leq n \}$.
 \end{proof}
 
 \begin{cor}\label{e-1} Let $H$ be an almost symmetric numerical semigroup. If
 $\e(H) = \m(H) - 1$, then $H^*$ is an almost symmetric with $\t(H^*) \geq 2$.
 \end{cor}
 
 \begin{proof} Assume that $H$ is almost symmetric. By Proposition \ref{type-H*} (2), if $e(H) = \m(H) - 1$, then
 $\t(H) + \t(H^*) = \m(H)$. We see from Theorem \ref{m:t+t*} that $H^*$ is almost symmetric.
 \end{proof}
 
 The converse of Corollary \ref{e-1} is not known.
 But if we assume that $H$ is symmetric, then that is true.
 
 \begin{cor} Let $H$ be a symmetric numerical semigroup with $\e(H) < \m(H)$. Then
 $\e(H) = \m(H) - 1$ if and only if $H^*$ is an almost symmetric with $\t(H^*) \geq 2$.
 \end{cor}
 
 \begin{proof} From Corollary \ref{e-1}, it is enough to show that $H^*$ is an almost symmetric 
 with $\t(H^*) \geq 2$, then $\e(H) = \m(H) - 1$. We assume that $H$ is symmetric and $H^*$ is almost symmetric  with $\t(H^*) \geq 2$.
 Then by Proposition \ref{type-H*}, we get $\t(H^*) = \e(H)$.
 On the other hand, using Theorem \ref{m:t+t*}, we have $\t(H) + \t(H^*) = 1 + \t(H^*) = \m(H)$.
 Hence $\e(H) = \m(H) - 1$.
 \end{proof}

 \section{proportionally modular numerical semigroups}
 \begin{defn} \cite{RGGU}
 A {\it proportionally modular Diophantine inequality}
 is an expression of the form $ax \mod b\leq cx$, where $a$, $b$ and $c$ are positive integers.
 We denote by $\S(a, b, c)$ the set of all integer solutions to this inequality.
 \end{defn}
 
 The set $\S(a, b, c)$ is a numerical semigroup (see \cite{RG}).
 
 \begin{defn} \cite{RGGU}
 A numerical semigroup $H$ is {\it proportionally modular} if it is the set of all integer solutions
 of a proportionally modular Diophantine inequality, that is, $H=\S(a, b, c)$ for some positive integers $a, b$ and $c$.
 \end{defn}
 
 Let $I$ be a closed interval and let $\langle I \rangle$ be a submonoid of $\mathbb{R}_{\geq 0}$ generated by closed interval $I$.
 We put $\S(I)=\langle I \rangle \cap \mathbb{N}$.
 It is easy to check that $\S(I)$ is a numerical semigroup.
 We call that $\S(I)$ is the numerical semigroup associated to $I$.
 It is known that every proportionally modular numerical semigroup can be realized 
 as the numerical semigroup associated to a closed interval. \par
 Our aim in this section is to give a formula for multiplicity of an opened modular numerical semigroups.
 As usual, for a rational number $r$, $\lfloor r \rfloor$ denotes the largest integer not bigger than $r$.
  
 \begin{prop}\cite{RGGU}\label{interval} Let $a, b$ and $c$ be a positive integers with $c<a$. Then
 \[\S(a, b, c)=\S\bigg(\bigg[\dfrac{b}{a},\dfrac{b}{a-c}\bigg]\bigg).\] 
 Conversely, every numerical semigroup associated to a closed interval is proportionally modular.
 \end{prop}
 
 A characterization of minimal generators of $\S(a, b, c)$ is given in \cite{RGU2}. 
 
 \begin{thm}\cite{RGU2} \label{pro-mod}Let $H$ be a numerical semigroup with $\e(H)=n$. 
 Then $H$ is proportionally modular if and only if for some rearrangement of its generators $\{a_1, a_2, ..., a_n\}$ 
 the following conditions hold:\par
 \begin{enumerate}
 \item $\gcd(a_i, a_{i+1})=1$ for all $i \in \{1, 2, ..., n-1\}$ \par
 \item $a_{i-1}+a_{i+1}\equiv 0$ $\mod a_i$ for all $i \in \{2, 3, ..., n-1\}$. 
 \end{enumerate}
 \end{thm}
 
 The Frobenius number of proportionally modular numerical semigroup has been computed in \cite{DR}.
 
 \begin{thm}\cite{DR}\label{frob of modular} Let $a, b$ and $c$ be a positive integers with $c< a< b$. Then
 \[\F(\S(a, b, c))=b-\bigg\lfloor \frac{\delta b}{a} \bigg\rfloor -1\] where 
 $\delta=\min\big\{k \in \{1, 2, ..., a-1\} \mid kb \mod a +\big\lfloor\frac{kb}{a}\big\rfloor c >(c-1)b+a-c \big\}$.
 \end{thm}
 
 Next we consider proportionally modular numerical semigroups $\S(a, b, 1)$.
 
 \begin{defn} \cite{RU}
 A {\it modular Diophantine inequality} is an expression of the form $ax \ \mod b \leq x$, with $a$ and $b$ positive integers.
 A numeical semigroup is {\it modular} if it is the set of integer solutions of a modular Diophantine inequality.
 \end{defn} 
 
 By Proposition \ref{interval}, modular numerical semigroup $\S(a, b, 1)$ 
 is determined by closed interval $[\frac{b}{a}, \frac{b}{a-1}]$.\par
 Recall that a numerical semigroup of the form $\{0, m, m+1, m+2, \ldots\}$ with a positive integer $m \geq 1$ is called {\it a half-line}.
 
 \begin{defn} \cite{RU}
 A numerical semigroup is {\it opened modular} if it is either a half-line or $H=\S(]\frac{b}{a}, \frac{b}{a-1}[)$ for some
 integers $a$ and $b$ with $2\leq a\leq b$.
 \end{defn}
 
 \begin{thm}\cite{RU}\label{properties of modular} Let $H=\S(]\frac{b}{a}, \frac{b}{a-1}[)$ and let $d=\gcd(a, b)$ and $d'=\gcd(a-1, b)$.
 Then the following conditions hold:
 \begin{enumerate}
 \item $\F(H)=b$, \par
 \item $\g(H)=\frac{b+d+d'-1}{2}$, \par
 \item $\t(H)=d+d'-1$, \par
 \item $\S([\frac{b}{a}, \frac{b}{a-1}])=H \cup \PF(H)$.
 \end{enumerate}
 \end{thm}
 
 By Theorem \ref{properties of modular} (4), we obtain the following.
 
 \begin{cor} \label{open*}Let $H=\S(]\frac{b}{a}, \frac{b}{a-1}[)$ be an opened modular numerical semigroup. Then
 \[H^*=\S\bigg(\bigg[\dfrac{b}{a},\dfrac{b}{a-1}\bigg]\bigg).\]
 \end{cor}
 
 Since $\g(H^*)=\g(H)-t(H)$, we have the following.
 
 \begin{thm}\cite{RGU1}\label{genus of modular} Let $H=\S([\frac{b}{a}, \frac{b}{a-1}])$ for some integers $0\leq a<b$. 
 We put $d=\gcd(a, b)$ and $d'=\gcd(a-1, b)$. Then
 \[\g(H)=\dfrac{b+1-d-d'}{2}.\]
 \end{thm}
 
 From Proposition \ref{F(H^*)} and Theorem \ref{frob of modular}, we obtain a formula for multiplicity of $\S(]\frac{b}{a}, \frac{b}{a-1}[)$
 in terms of $a$ and $b$.
 
 \begin{thm} \label{mul-opened} Let $H=\S(]\frac{b}{a}, \frac{b}{a-1}[)$ for some integers $0\leq a<b$. Then
 \[\m(H)=\bigg\lfloor \dfrac{\delta b}{a} \bigg\rfloor +1,\]
 where $\delta=\min\big\{k \in \{1, 2, ..., a-1\} \mid kb \mod a +\big\lfloor\frac{kb}{a}\big\rfloor c >(c-1)b+a-c \big\}$.
 \end{thm}
 
 \begin{proof} By Proposition \ref{F(H^*)}, we get $\m(H)=\F(H)-\F(H^*)$.
 We obtain the desired formula by using Theorem \ref{frob of modular} and \ref{properties of modular}.
 \end{proof}
 
 \begin{exam} Let $H = \S \left(]\frac{11}{5}, \frac{11}{4}[ \right)$.
 From Theorem \ref{properties of modular}, $\F(H) = 11$, $\g(H) = 6$ and $\t(H) = 1$, hence $H$ is symmetric.
 Also we have $\delta = 2$.
 By Theorem \ref{mul-opened}, $\m(H) = \lfloor \frac{2\cdot 11}{5} \rfloor + 1 = 5$.
 Indeed, by direct computing we obtain that $H = \left< 5, 7, 8, 9 \right>$ and $\PF(H) = \{11\}$.
 Furthermore $H^* = H \cup \PF(H) = \left< 5, 7, 8, 9, 11 \right> = \S \left([\frac{11}{5}, \frac{11}{4}] \right)$, 
 from Theorem \ref{properties of modular} and Corollary \ref{open*}.
 \end{exam}
 
 Let $H$ be an opened modular numerical semigroup.
 Then $2\g(H) = \F(H) + \t(H)$, from Theorem \ref{properties of modular}.
 Hence we have the following,  using Theorem \ref{almsymm}.
 \begin{cor}\label{open is alm} Opened modular numerical semigroups are almost symmetric.
 \end{cor}
 
 \section{Proportionally modular numerical semigroups generated by three elements}
 In this section $H = \langle a, b, c \rangle$ will represent a proportionally modular numerical semigroup generated by three elements.
 From Theorem \ref{pro-mod}, we can assume that  
 $\gcd(a, b) = \gcd(b, c) = 1$ and $db = a + c$ for some $d \geq 2$.
 
 \begin{thm}\cite{RGU2} Let $H = \langle a, b, c \rangle$ be a proportionally modular numerical semigroup. Then
 $H$ is symmetric if and only if $d = \gcd(a, c)$. Moreover, if $H$ is symmetric, then
 \begin{enumerate}
 \item $\F(H) = \frac{abc - ab - bc}{a + c}$, \par
 \item $\g(H) = \frac{abc - ab - bc + a + c}{2(a + c)}$.
 \end{enumerate}
 \end{thm}  
 
 We now let $\varphi: k[X,Y,Z]\rightarrow k[ H ]=k[t^a,t^b,t^c]$ 
the $k$-algebra homomorphism defined by $\varphi (X)=t^a$, $\varphi (Y)=t^b$, 
and $\varphi (Z)=t^c$ and let  $\fkp=\fkp (a,b,c)$ be 
 the kernel of $\varphi$.  Then it is known that if $H$ is not symmetric, 
then the ideal  $\fkp=\Ker(\varphi)$ is 
 generated by the maximal minors of the matrix 
\begin{align} \label{matrix}
\left( 
\begin{array}{lll}
X^\alpha   &  Y^\beta    &      Z^\gamma  \\
Y^{\beta'} &  Z^{\gamma'}&      X^{\alpha'}
\end{array}
\right) \tag{$*$}
\end{align}

Observing the matrix (\ref{matrix}), we have the following. 

\begin{align*}
\left(
\begin{array}{ccc}
\alpha + \alpha ' & - \beta'       & - \gamma \\
- \alpha          & \beta + \beta' & - \gamma' \\
- \alpha'         & - \beta        & \gamma + \gamma' 
\end{array}
\right)
\left(
\begin{array}{c}
a \\
b \\
c
\end{array}
\right) =
\left(
\begin{array}{c}
0 \\
0 \\
0
\end{array}
\right).
\end{align*} 

It is easy to show that 
\begin{align}\label{abc}
a & = \beta \gamma + \beta '\gamma + \beta '\gamma ', \notag \\
b & = \gamma \alpha + \gamma' \alpha + \gamma '\alpha ', \tag{$**$} \\
c & = \alpha \beta + \alpha '\beta + \alpha '\beta '.\notag
\end{align} 
Then $\PF(H) = \{\alpha a+(\gamma +\gamma ')c-(a+b+c), \beta 'b+(\gamma +\gamma ')c-(a+b+c)\}$ (see \cite{NNW}).

\begin{thm}\cite{NNW} \label{ps-mat}
 Let $H = \langle a, b, c \rangle$ be a numerical semigroup. Then
 $H$ is pseudo-symmetric if and only if $\alpha \beta \gamma = 1$ or $\alpha '\beta '\gamma '= 1$.
\end{thm}

We assume that 
$H = \langle a, b, c \rangle$ is not symmetric proportionally modular numerical semigroup.
Then the matrix (\ref{matrix}) is
\[\left(
\begin{array}{lll}
X          &  Y^\beta  &  Z^\gamma   \\
Y^{\beta'} &  Z        &  X^{\alpha'}
\end{array}
\right).\]

By Theorem \ref{ps-mat}, we have the following results.

\begin{cor} Let $H = \langle a, b, c \rangle$ be a proportionally modular numerical semigroup. 
Then the following conditions are equivalent.
\begin{enumerate}
\item $H$ is pseudo-symmetric, \par
\item $\beta \gamma = 1$ or $\alpha' \beta' = 1$,\par
\item $d = \frac{a + 1}{2}$ or $\frac{c + 1}{2}$.
\end{enumerate}
\end{cor}

\begin{proof} By Theorem \ref{ps-mat}, the equivalence of (1) and (2) is obvious.\par
$(2) \Longrightarrow (3)$. We assume that (2) is hold. Then the matrix (\ref{matrix}) is
\[\left(
\begin{array}{lll}
X          &  Y  &  Z   \\
Y^{\beta'} &  Z        &  X^{\alpha'}
\end{array}
\right) \ 
or \
\left(
\begin{array}{lll}
X          &  Y^\beta  &  Z^\gamma   \\
Y &  Z        &  X
\end{array}
\right),\]
hence we obtain that $d = 1 + \beta'$ or $1 + \beta$.
By equation (\ref{abc}), this shows that $d = \frac{a + 1}{2}$ or $d = \frac{c + 1}{2}$.\par
$(3) \Longrightarrow (2)$. It suffices to prove that if $d =\frac{a + 1}{2}$, then $\beta \gamma  = 1$.
We assume $d =\frac{a + 1}{2}$. Since $\alpha = \gamma' = 1$, we have that
\begin{align*}
a  &= \beta \gamma + \beta '\gamma + \beta ', \\
b  &= \gamma + \alpha' + 1, \\
c  &= \beta + \alpha' \beta + \alpha'\beta'.
\end{align*}
From this equations,  
\[d = \beta + \beta' = \frac{\beta \gamma + \beta' \gamma + \beta' + 1}{2}.\]
Hence we get $\beta \gamma = 1$.
\end{proof}

\begin{cor} Let $H = \langle a, b, c \rangle$ be a proportionally modular and pseudo-symmetric numerical semigroup. 
Then by rearrange of its generators $\{a, c\}$ we have that
\begin{enumerate}
\item $\F(H) = 2(c - b)$, \par
\item $\g(H) = c - b + 1$. 
\end{enumerate}
\end{cor}

\begin{exam} Let $H = \langle 5, 7, 16 \rangle$.
From Theorem \ref{pro-mod}, $H$ is proportionally modular numerical semigroup.
Then the matrix (\ref{matrix}) is
$\left(
\begin{array}{lll}
X    &  Y  &  Z \\
Y^2  &  Z  &  X^5
\end{array}
\right)$.
Hence we obtain that $H$ is pseudo-symmetric and $\PF(H) = \{9, 18 \}$.
\end{exam}

 \section{Gluing of numerical semigroups}
 
 The concept of gluing of numerical semigroups was defined by \cite{De} and \cite{Ro}.
 
 \begin{defn} \label{gluing} \cite{De}, \cite{Ro} Let $H_1=\langle a_1, a_2, \ldots, a_n \rangle$
 and $H_2=\langle b_1, b_2, \ldots, b_m \rangle$ be two numerical semigroups.
 Take $y \in H_1 \setminus \{a_1, a_2, \ldots, a_n \}$ and $x \in H_2 \setminus \{b_1, b_2, \ldots, b_m \}$
 such that $(x, y)=1$.
 We say that 
 \[H=\langle xH_1, yH_2 \rangle = \langle xa_1, xa_2, \ldots, xa_n, yb_1, yb_2, \ldots, yb_m \rangle \]
 is a {\it gluing} of $H_1$ and $H_2$.
 \end{defn}
 
 \begin{thm} \label{gl-of-ci} \cite{De}, \cite{Ro} Let $H$ be a numerical semigroup other than $\mathbb{N}$. Then 
 $H$ is complete intersection if and only if $H$ is a gluing of two complete intersection numerical semigroups.
 \end{thm}
 
 The symmetry is preserved under gluing.
 
 \begin{thm} \cite{De}, \cite{Ro} A gluing of symmetric numerical semigroups is symmetric. Therefore,
 every numerical semigroup that is a complete intersection is symmetric.
 \end{thm}
 
 \begin{lemma}\label{f-f_1} Let $H$ be a numerical semigroup and let $\PF(H)=\{f_1< \cdots <f_{\t(H) - 1}< \F(H) \}$.
 Then 
 \[\F(H)-f_1 \leq f_{\t(H)-1}.\] 
 \end{lemma}
 
 \begin{proof} Since $\F(H) - f_1 \in \PF(H)$ implies $\F(H) - f_1 = f_{\t(H) - 1}$,
 we assume that $\F(H) - f_1 \not \in \PF(H)$.
 Then there exist $0 \not = h \in H$ such that $\F(H) - f_1 + h \in \PF(H)$.
 Hence we have that $\F(H)-f_1 \leq f_{\t(H)-1}$.
 \end{proof}
 
 Now, let $H=\langle xH_1, yH_2 \rangle$ be a gluing of $H_1$ and $H_2$ with $(x, y)=1$ and $xy \in xH_1 \cap yH_2$.\par
 The following is the key lemma to calculate the pseudo-Frobenius numbers.
 
 \begin{lemma}\label{gl-ap} If $H=\langle xH_1, yH_2 \rangle$ as above, then
 \[\Ap(H, xy)=\{xs + yt \mid s \in \Ap(H_1, y), t \in \Ap(H_2, x) \}. \]
 \end{lemma}
 
 \begin{proof} If $s \in \Ap(H_1, y)$ and $t \in \Ap(H_2, x)$, then $s-y \not \in H_1$ and $t-x \not \in H_2$, thus
 we obtain $xs + yt - xy = x(s-y) + y(t-x) + xy \not \in H$.
 Also, the cardinality of $\{xs + yt \mid s \in \Ap(H_1, y), t \in \Ap(H_2, x) \}$ is equal to $xy$.
 Hence we have that $\Ap(H, xy)=\{xs + yt \mid s \in \Ap(H_1, y), t \in \Ap(H_2, x) \}$.
 \end{proof}
 
 By using Lemma \ref{gl-ap}, we can calculate pseudo-Frobenius numbers of $H$.
 
 \begin{prop}\label{gl-pf} If $H=\langle xH_1, yH_2 \rangle$ as above, then
 \[\PF(H)=\{ xf + yf' + xy \mid f \in \PF(H_1), f'\in \PF(H_2) \}\]
 with $\t(H)=\t(H_1)\t(H_2)$. 
 \end{prop}
 
 \begin{proof} Clearly $x\PF(H_1) + y\PF(H_2) + xy \subseteq \PF(H)$.
 We take $f \in \PF(H)$.
 From Lemma \ref{gl-ap}, there exist $s \in \Ap(H_1, y)$ and $t \in \Ap(H_2, x)$ such that 
 $f = x(s-y) + y(t - x) + xy$.
 It suffices to prove that $s - y \in \PF(H_1)$ and $t - x \in \PF(H_2)$.
 If $s - y \not \in \PF(H_1)$, then there exists $h \in H_1$ such that $s - y + h \not \in H_1$,
 that is, $x(s- y + h) \not \in xH_1 \subset H$.
 But this lead to $H \ni f + xh = x(s- y + h) + y(t - x) + xy \not \in H$, which is impossible.
 \end{proof}
 
 Hence we have 
 \[\F(H)=x\F(H_1)+y\F(H_2)+xy.\]

 \begin{thm}\label{gl-notalm} Let $H_1$ and $H_2$ be two numerical semigroups.
 Assume $H_1$ or $H_2$ is not symmetric. Then the gluing of $H_1$ and $H_2$ is not almost symmetric. 
 \end{thm}
 
 \begin{proof} Set $\PF(H_1) = \{f_1 < \cdots < f_{\t(H_1) - 1} < \F(H_1) \}$,
 $\PF(H_2) = \{f_1' < \cdots < f_{\t(H_2) - 1}' < \F(H_2) \}$
 and $\PF(H) = \{g_1 < \cdots < g_{\t(H) - 1} < \F(H) \}$.
 We can assume $\F(H_1) > \F(H_2)$.
 From Lemma \ref{f-f_1} and Proposition \ref{gl-pf}, we have that
 \[\F(H) - g_1 \leq g_{\t(H)-1} = x\F(H_1) + yf_{\t(H_2)-1}' + xy \]
 and
 \[g_1 = xf_1 + yf'_1 + xy.\]
 If $\F(H) - g_1 = g_{\t(H)-1}$, then we get
 \[\F(H) = x(f_1 + \F(H_1)) + y(f'_1 + f'_{\t(H_2) - 1} + x) + xy.\]
 By hypothesis, $f_1 + \F(H_1) \in H_1$ and $f'_1 + f'_{\t(H_2) - 1} + x \in H_2$.
 Namely, $\F(H) \in \left< xH_1, yH_2 \right> = H$. This is a contradiction.
 Hence we obtain that $\F(H) - g_1 < g_{\t(H)-1}$.
 From Theorem \ref{almsymm}, this shows that the symmetry of $\PF(H)$ does not hold and hence $H$ is not almost symmetric.
 \end{proof}
 
 \begin{exam} (1) We set $H_1= \langle6, 10, 11, 13, 14 \rangle$ and $H_2=\langle 7, 8, 10, 13 \rangle$.
 Then $\PF(H_1)=\{7, 8, 15\}$ and $\PF(H_2)=\{19\}$.
 Take $x=14$ and $y=17$. We see from Proposition \ref{gl-pf} that 
 $H=\langle 14H_1, 17H_2 \rangle = \langle 84,119,136,140,154,170,182,196,221 \rangle$ and $\PF(H)=\{659, 673, 771\}$.
 Hence $H$ is not almost symmetric.\\
 (2) Let $T = \left< 3, b, c \right>$ be a numerical semigroup with $3 < b < c$.
 Since $T$ has maximal embedding dimension, we have that $T$ is pseudo-symmetric if and only if $c = 2b - 3$.
 We assume that $T$ is pseudo-symmetric.
 Taking $k \in T \setminus \{3, b, c\}$, we put $H = \left<2T, k \right> = \left< 6, 2b, 2c, k \right>$,
 that is, $H$ is the gluing of $T$ and $\mathbb{N}$.
 By Proposition \ref{gl-pf}, we get $\PF(H)= \{2b + k - 6 < 2c + k - 6\}$.
 This means 
 \[\Ap(H, 6) = \{0, 2b, 2c, k, 2b + k, 2c + k\}.\]
 Since $H^*$ has maximal embedding dimension and $c = 2b - 3$, we get
 \[\Ap(H^*, 6) = \{0, 2b, 2c, k, 2b + k - 6, 2c + k - 6\},\]
 and
 \begin{align*}
 \PF(H^*) &= \{2b - 6, 2c - 6, k - 6, 2b + k -12, 2c + k - 12\} \\
          &= \{2b - 6, 4b - 12, k - 6, 2b + k - 12, 4b + k - 18\}
 \end{align*}
 
 Hence $H^*$ is almost symmetric.
 
 \end{exam}
 
 \begin{acknowledgement}
 The author would like to thank professor Kei-ichi Watanabe for suggesting the problem and for useful discussion.
 \end{acknowledgement}

\end{document}